\documentclass[11pt,reqno]{amsart}
\usepackage{amssymb,amsthm,thm-restate,mathtools,tikz,manyfoot,enumerate,appendix}%
\usepackage[foot]{amsaddr}
\usepackage{charter}
\usepackage{euler}
\usepackage[T1]{fontenc}
\usepackage{geometry}
\usetikzlibrary{arrows,snakes,backgrounds,calc,math}
\tikzstyle{map}=[->,semithick]
\tikzstyle{arc}=[bend left,->,semithick]
\tikzstyle{rinclusion}=[right hook->,semithick]
\tikzstyle{linclusion}=[left hook->,semithick]
\def\S{{\mathbb S}}
\def\Ri{\mathcal R}
\def\M{{\mathcal M}}
\def\R{{\mathbb R}}
\def\B{{\mathbb B}}
\def\C{\mathcal C}
\def\K{{\mathcal K}}
\def\L{{\mathcal L}}
\newcommand{\defref}[1]{Definition~\ref{def:#1}}

\newcommand{\lemref}[1]{Lemma~\ref{lem:#1}}
\newcommand{\remref}[1]{Remark~\ref{rem:#1}}
\newcommand{\propref}[1]{Proposition~\ref{prop:#1}}
\newcommand{\secref}[1]{Section~\ref{sec:#1}}
\newcommand{\subsecref}[1]{Subsection~\ref{subsec:#1}}
\newcommand{\thmref}[1]{Theorem~\ref{thm:#1}}
\newcommand{\eqnref}[1]{(\ref{eqn:#1})}
\newcommand{\diam}[2][\M]{\mathrm{diam}_{#1}\left(#2\right)}
\newcommand{\eqdef}{\stackrel{def}{=}}%
\newcommand{\supp}[1]{\mathrm{supp}{#1}}
\newcommand{\os}[1]{\left\langle #1 \right\rangle}
\newcommand{\sd}[2][]{\mathrm{sd}^{#1}\left(#2\right)}
\newcommand{\norm}[1]{\left\lVert #1 \right\rVert}
\renewcommand{\mod}[1]{\big\lvert #1 \big\rvert}
\newcommand{\bc}[1]{\widehat{#1}}
\newcommand{\map}[1][] {\xrightarrow{\quad#1\quad}}

\theoremstyle{plain}%
\newtheorem{theorem}{Theorem}[section]
\newtheorem{proposition}[theorem]{Proposition}%
\newtheorem{lemma}[theorem]{Lemma}

\theoremstyle{definition}
\newtheorem{definition}[theorem]{Definition}

\theoremstyle{remark}
\newtheorem{remark}[theorem]{Remarks}%
\author{Sushovan Majhi}
\address{George Washington University, Washington D.C., USA}
\email{s.majhi@gwu.edu}
\title{Demystifying Latschev's Theorem: Manifold Reconstruction from Noisy Data}


\begin{document}
	
\begin{abstract}
For a closed Riemannian manifold $\mathcal{M}$ and a metric space $S$ with a
small Gromov--Hausdorff distance to it, Latschev's theorem guarantees the
existence of a sufficiently small scale $\beta>0$ at which the Vietoris--Rips
complex of $S$ is homotopy equivalent to $\mathcal{M}$. Despite being regarded
as a stepping stone to the topological reconstruction of Riemannian manifolds
from noisy data, the result is only a qualitative guarantee. Until now, it had
been elusive how to quantitatively choose such a proximity scale $\beta$ in
order to provide sampling conditions for $S$ to be homotopy equivalent to
$\mathcal{M}$. In this paper, we prove a stronger and pragmatic version of
Latschev's theorem, facilitating a simple description of $\beta$ using the
sectional curvatures and convexity radius of $\mathcal{M}$ as the sampling
parameters.
Our study also delves into the topological recovery of a closed Euclidean
submanifold from the Vietoris--Rips complexes of a Hausdorff close Euclidean
subset. As already known for \v{C}ech complexes, we show that Vietoris--Rips
complexes also provide topologically faithful reconstruction guarantees for
submanifolds. 
\end{abstract}

\maketitle

\section{Introduction}
Given a metric space $(X,d_X)$ and a positive proximity scale $\beta$, the
\emph{Vietoris--Rips complex} of $X$, denoted $\Ri_\beta(X)$, is defined to be
an abstract simplicial complex having an $m$--simplex for every finite subset of
$X$ with cardinality $m+1$ and diameter less than $\beta$. 

The notion was first introduced by L. Vietoris \cite{Vietoris1927}, then
extensively studied by E. Rips in the context of hyperbolic groups. Despite its
inception in the early twentieth century, it is only the last decade that
witnessed an increasing popularity of these complexes---particularly in the
applied topology and topological data analysis (TDA) community.  The
computational simplicity makes the Vietoris--Rips complexes more palatable in
applications than its traditional alternatives, e.g., the \v{C}ech complexes.

The combinatorial flexibility, however, comes at a theoretical cost. The
topology of the Vietoris--Rips complex of (even a finite) metric space is
generally poorly understood. Nonetheless, there have been noteworthy
developments in the study of the Vietoris--Rips complexes constructed on or near
Riemannian manifolds \cite{hausmann_1995,latschev_2001} and metric graphs
\cite{Majhi2023}.

\subsection{Motivation}
In his pioneering work \cite{hausmann_1995}, Hausmann established a homotopy
equivalence between a closed Riemannian manifold $\M$ and its Vietoris--Rips
complex at a scale $\beta$ smaller than the convexity radius (\defref{conv-rad})
of $\M$. As a naive exercise \cite[Problem 3.11]{hausmann_1995}, Hausmann asked
the curious question of \emph{finite reconstruction}: for a dense enough subset
$S\subset\M$ and small enough scale $\beta$, is $\M$ also homotopy equivalent to
the Vietoris--Rips complex of $S$?

The quest foreshadowed Latschev's remarkable result \cite[Theorem
1.1]{latschev_2001}: for every closed Riemannian manifold $\M$, there exists a
positive number $\epsilon_0$ such that for any $0<\beta\leq\epsilon_0$ there
exists some $\delta>0$ such that for every metric space $S$ with
Gromov--Hausdorff distance to $\M$ less than $\delta$, the Vietoris--Rips
complex $\Ri_\beta(S)$ is homotopy equivalent to $\M$. The result is only
qualitative in nature, guaranteeing the existence of a sufficiently small scale
$\beta$ such that $\M$ is homotopy equivalent to the Vietoris--Rips complex of a
metric space $(S,d_S)$ that is close to $\M$ in the Gromov--Hausdorff distance
(\defref{gh}). The result has been regarded as a stepping stone to the finite
reconstruction of an abstract Riemannian manifold from a noisy sample. Despite
the qualitative guarantee, it is not apparently clear how to quantitatively
choose such an $\epsilon_0$ for a given manifold $\M$. 

The current paper primarily aims at presenting the first quantitative version
(\thmref{gh-hom}) of Latschev's theorem in order to develop a provable
Vietoris--Rips inspired manifold reconstruction scheme from a noisy sample. Our
sampling conditions---for a faithful reconstruction---are given based on the
convexity radius (\defref{conv-rad}) and the upper bound of the sectional
curvatures of $\M$. We recognize these parameters to be very natural in the
context of Riemannian manifolds---strict enough to prove the desired homotopy
equivalences and flexible enough to retain practicality. 

Our techniques naturally extend to the topological reconstruction of a Euclidean
submanifold from a Hausdorff--close sample. For a sufficiently small scale, we
prove in (\thmref{h-hom}) that a submanifold is homotopy equivalent to the
Vietoris--Rips complex of a dense sample. In the Euclidean case, we describe the
sampling conditions using the reach (\defref{reach}) of the submanifold.

\subsection{Related Work}
In the same vein, Majhi \cite{Majhi2023} has recently studied the Vietoris--Rips
complexes near a special class of geodesic spaces: metric graphs. For the
Vietoris--Rips complexes of a noisy sample, the author provides sampling
conditions for a faithful topological recovery of metric graphs under both the
Gromov--Hausdorff and Hausdorff noise. Our investigation revolves around the
same reconstruction theme, but for Riemannian manifolds. Since graphs can
generally have branches and non-smooth corners, it is worth noting that the
relevance of the results of Majhi \cite{Majhi2023} on metric graphs are not
subsumed by our results on manifolds.

Relevant in this context are the works of Adams et al.
\cite{Adamaszek2018,Adams_2019}, where the homotopy equivalence results of
Hausmann \cite{hausmann_1995} and Latschev \cite{latschev_2001} have been
restated in terms of the Vietoris--Rips thickening via the theory of optimal
transport. We also mention \cite{Virk2021-lc} and \cite{Lemez2022-au} for
providing an alternative and much simpler proof of Hausmann's theorem and
extending Latschev's result to selective Rips complexes, respectively.

In this paper, we also further our understanding of the Vietoris--Rips complexes
in another fundamental direction: the recovery of a Euclidean submanifold
$\M\subset\R^d$ from a noisy Euclidean sample $S\subset\R^d$ in its close
Hausdorff proximity. There have been several approaches to recover the topology
(sometimes only the homology/homotopy groups) of a submanifold using the
\v{C}ech and Vietoris--Rips complexes and their filtrations, as we survey in the
following paragraph.

In a landmark paper by Niyogi et al. \cite{SMALE}, the normal injectivity radius
of $\M$ has been recognized as the sampling parameter; for a sufficiently small
scale, the \v{C}ech complexes of a dense sample $S\subset\M$ have been shown to
be homotopy equivalent to $\M$. For the recovery of homology/homotopy groups of
Euclidean shapes using filtrations of the \v{C}ech and Vietoris--Rips complexes,
we mention the works of Chazal et al. \cite{chazal2014persistence,CHAZALSTAB}
using the weak feature size of compact sets and Majhi et al.
\cite{fasy2018reconstruction} using the convexity radius and distortion of
geodesic subspaces. For a submanifold $\M$, we ask a more ambitious inference
question: are the Vietoris--Rips complexes of a Hausdorff--close sample homotopy
equivalent to $\M$ for a sufficiently small scale? 

Attali et al. \cite[Theorem 14]{ATTALI2013448} and Kim et al. \cite[Theorem
20]{kim2020homotopy} show that a Euclidean subset with a positive reach (more
generally $\mu$--reach) is homotopy equivalent to the Vietoris--Rips complexes
of a noisy point-cloud. Our \thmref{h-hom} puts forward a similar result for
submanifold reconstruction in terms of the reach of $\M$. Although, some results
in \cite{ATTALI2013448,kim2020homotopy} generalize and improve over
\thmref{h-hom}. 

\subsection{Our Contribution}
One of the major contributions of this work is to quantify the scale parameters
at which the Vietoris--Rips complex of a sample $S$ recovers (up to homotopy
type) a Riemannian manifold $\M$, under both the Gromov--Hausdorff and Hausdorff
sampling conditions. Our main homotopy equivalence results are presented in
\thmref{gh-hom} and \thmref{h-hom}, respectively. 

This paper is organized in the following manner. \secref{prelim} contains
definitions, notations, and facts that are frequently used throughout the paper.
In \secref{abstract}, the recovery of an abstract Riemannian manifold $\M$ from
a Gromov--Hausdorff close sample $S$ is obtained. We present a novel proof of
Latschev's theorem with a much stronger statement. The sampling parameter
$\Delta(\M)$ as defined in \eqnref{delta}.
\begin{restatable*}[Manifold Reconstruction under Gromov--Hausdorff Distance]
{theorem}{ghhom}\label{thm:gh-hom} Let $(\M,d_\M)$ be a closed, connected
Riemannian manifold. Let $(S,d_S)$ be a compact metric space and $\beta>0$ a
number such that
\[
	\frac{1}{\zeta}d_{GH}(\M,S)<\beta<\frac{1}{1+2\zeta}\Delta(\M)
\] for some $0<\zeta\leq1/14$. Then, $\mod{\Ri_\beta(S)}\simeq\M$.
\end{restatable*}
\secref{euclidean} is devoted to the recovery of a Euclidean submanifold
$\M\subset\R^d$ from a Hausdorff close, Euclidean sample $S\subset\R^d$.
\thmref{h-hom} shows the homotopy equivalence between $\M$ and the
Vietoris--Rips complex $\Ri_\beta(S)$. Here, $\tau(\M)$ denotes the reach
(\defref{reach}) of $\M$.
\begin{restatable*}[Submanifold Reconstruction under Hausdorff Distance]
{theorem} {hhom}\label{thm:h-hom} Let $\M\subset\R^d$ be a closed, connected
Euclidean submanifold. Let $S\subset\R^d$ be a compact subset and $\beta>0$ a
number such that 
\[
	\frac{1}{\zeta}d_{H}(\M,S)<\beta\leq\frac{3(1+2\zeta)(1-14\zeta)}{8(1-2\zeta)^2}\tau(\M)
\] 
for some $0<\zeta<1/14$. Then, $\mod{\Ri_\beta(S)}\simeq\M$.
\end{restatable*}

\section{Preliminaries}\label{sec:prelim} In this section, we present
definitions and notations that we use throughout the paper. The standard results
from algebraic topology and Riemannian geometry are stated here without proof;
details can be found in any standard textbook on the subjects, e.g.,
\cite{MUNK,spanier1994algebraic} and
\cite{berger2007panoramic,BishopRichardL2001Gom}, respectively. 

\subsection{Metric Spaces}
Let $(X,d_X)$ be a metric space. When it is clear from the context, we omit the
metric $d_X$ from the notation, and denote the metric space just by $X$. 
\begin{definition}[Diameter]\label{def:diam} The diameter, denoted
$\diam[X]{Y}$, of a subset $Y\subset X$ is defined by the supremum of the
pairwise distances in $Y$.
\[
	\diam[X]{Y}\eqdef\sup_{y_1,y_2\in Y}d_X(y_1,y_2).
\]
\end{definition}
When $Y$ is compact, its diameter is finite. We denote by $\B_X(x,r)$ the metric
ball of radius $r\geq0$ centered at $x\in X$.

A \emph{correspondence} $\C$ between two (non-empty) metric spaces $(X,d_X)$ and
$(Y,d_Y)$ is defined to be a subset of $X\times Y$ such that
\begin{enumerate}[(a)]
\item for any $x\in X$, there exists $y\in Y$ such that $(x,y)\in\C$, and
\item for any $y\in Y$, there exists $x\in X$ such that $(x,y)\in\C$.
\end{enumerate}
We denote the set of all correspondences between $X,Y$ by $\C(X,Y)$. The
\emph{distortion} of a correspondence $\C\in\C(X,Y)$ is defined as:
\[
\mathrm{dist}(\C)\stackrel{\text{def}}{=}\sup_{(x_1,y_1),(x_2,y_2)\in\C}
\mod{d_X(x_1,x_2)-d_Y(y_1,y_2)}.
\]
\begin{definition}[Gromov--Hausdorff Distance]\label{def:gh} Let $(X,d_X)$ and
$(Y,d_Y)$ be two compact metric spaces. The \emph{Gromov--Hausdorff} distance
between $X$ and $Y$, denoted by $d_{GH}(X,Y)$, is defined as:
\[
	d_{GH}(X,Y)\stackrel{\text{def}}{=}
	\frac{1}{2}\left[\inf_{\C\in\C(X,Y)}\mathrm{dist}(\C)\right].
\]
\end{definition}

\subsection{Simplicial Complexes}
An \emph{abstract simplicial complex}~$\K$ is a collection of finite sets such
that if  $\sigma\in\K$, then so are all its non-empty subsets. In general,
elements of $\K$ are called \emph{simplices} of $\K$. The singleton sets in $\K$
are called the \emph{vertices} of $\K$. If a simplex $\sigma\in\K$ has
cardinality $(m+1)$, then it is called an \emph{$m$-simplex} and is denoted by
$\sigma_m$. An $m$--simplex $\sigma_m$ is also written as
$[v_0,v_1,\ldots,v_m]$, where $v_i$'s belong to the vertex set of $\K$. If
$\sigma'$ is a (proper) subset of $\sigma$, then $\sigma'$ is called a (proper)
\emph{face} of $\sigma$, written as $\sigma'\preceq\sigma$ ($\sigma'\prec\sigma$
when proper).

Let $\K_1$ and $\K_2$ be abstract simplicial complexes with vertex sets $V_1$
and $V_2$, respectively. A \emph{vertex map} is a map between the vertex sets.
Let $\phi \colon V_1 \to V_2$ be a vertex map. We say that $\phi$ induces a
\emph{simplicial map} $\phi:\K_1\to\K_2$ if for all
$\sigma_m=[v_0,v_1,\ldots,v_m]\in K_1$, the image
\[
	\phi(\sigma_m)\eqdef[\phi(v_0),\phi(v_1),\ldots,\phi(v_m)]
\]
is a simplex of
$\K_2$. Two simplicial maps~$\phi,\psi:\K_1\to\K_2$ are called \emph{contiguous}
if for every simplex $\sigma_1\in\K_1$, there exists a simplex $\sigma_2\in\K_2$
such that~$\phi(\sigma_1)\cup\psi(\sigma_1)\preceq\sigma_2$. 

For an abstract simplicial complex $\K$ with vertex set $V$, one can define its
\emph{geometric complex} or \emph{underlying topological space}, denoted by
$\mod\K$, as the space of all functions $h:V\to[0,1]$ satisfying the following
two properties:
\begin{enumerate}[(i)]
	\item $\supp(h)\eqdef\{v\in V\mid h(v)\neq0\}$ is a simplex of
	$\K$, and 
	\item $\sum\limits_{v\in V}h(v)=1$.
\end{enumerate}
For $h\in\mod{\K}$ and vertex $v$ of $\K$, the real number $h(v)$ is called the
$v$--th \emph{barycentric coordinate} of $h$. For a simplex $\sigma$ of $\K$,
its \emph{closed simplex} $\mod{\sigma}$ and \emph{open simplex} $\os{\sigma}$
are subsets of $\mod{\K}$ defined as follows:
\[
	\mod{\sigma}\eqdef\left\{h\in\mod{\K}\mid\supp(h)\subseteq\sigma\right\},
	\text{ and } 
	\os{\sigma}\eqdef\left\{h\in\mod{\K}\mid\supp(h)=\sigma\right\}.
\]
In this work, we use the standard metric topology on $\mod{\K}$, as defined in
\cite{spanier1994algebraic}. A simplicial map $\phi:\K_1\to\K_2$ induces a
continuous (in this topology) map $\mod\phi:\mod{\K_1}\to\mod{\K_2}$ defined by
\[ \mod\phi(h)(v')\eqdef\sum\limits_{\phi(v)=v'}h(v),\text{ for }v'\in\K_2. \]
From the above definition, it follows that
$\mod{\phi}(h)\in\mod{h(\sigma)}$ whenever $h\in\os{\sigma}$.

A simplicial complex $\K$ is called a \emph{pure $m$--complex} if every simplex
of $\K$ is a face of an $m$--simplex. A simplicial complex $\K$ is called a
\emph{flag complex} if $\sigma$ is a simplex of $\K$ whenever every pair of
vertices in $\sigma$ is a simplex of $\K$.

\subsection{Barycentric Subdivision}\label{subsec:subdiv}
The \emph{barycenter}, denoted $\widehat{\sigma}_m$, of an $m$--simplex
$\sigma_m=[v_0,v_1,\ldots,v_m]$ of $\K$ is the point of $\os{\sigma_m}$ such
that $\widehat{\sigma}_m(v_i)=\frac{1}{m+1}$ for all $0\leq i\leq m$. Using
linearity of simplices, a more convenient way of writing this is: 
$$\widehat{\sigma}_m=\sum_{i=0}^m \frac{1}{m+1}v_i.$$ Let $\K$ be a complex. A
\emph{subdivision} of $\K$ is a simplicial complex $\K'$ such that 
\begin{enumerate}[(i)]
    \item the vertices of $\K'$ are points of $\mod{\K}$,
    \item if $s'$ is a simplex of $\K'$, then there is $s\in\K$ such that
        $s'\subset\mod{s}$, and
    \item the linear map $h:\mod{\K'}\to\mod{\K}$ sending each vertex of $\K'$
        to the corresponding point of $\mod{\K}$ is a homeomorphism.
\end{enumerate}
For a simplicial complex $\K$, its \emph{barycentric subdivision}, denoted by
$\sd{\K}$, is a special subdivision defined as follows. The vertices of
$\sd{\K}$ are the barycenters of the simplices of $\K$. The simplices of
$\sd{\K}$ are (non-empty) finite sets
$[\widehat\sigma_0,\widehat\sigma_1,\ldots,\widehat\sigma_m]$ such that
$\sigma_{i-1}\prec\sigma_i$ for $1\leq i\leq m$ and $\sigma_i\in\K$.

\subsection{Riemannian Manifolds}
Let $\M$ be an $n$--dimensional Riemannian manifold, equipped with the shortest
geodesic metric $d_\M$. More formally, for any two points $p,q\in \M$, their
distance is given by
\[
d_\M(x,y)\eqdef\inf\big\{\mathrm{length}(\gamma)\mid\gamma
\text{ is a smooth curve in }\M\text{ joining }p,q\big\}; 
\]
see \cite[p. 174]{berger2007panoramic} for more details. Throughout the paper,
we always assume that $\M$ is connected and closed (without boundary and
compact). A subset $A\subset\M$ is called \emph{(geodesically) convex} if for
any two points $p,q\in A$, there exists a unique minimizing geodesic segment
from $p$ to $q$ whose image lies entirely in $A$. 
\begin{definition}[Convexity Radius]\label{def:conv-rad} The \emph{convexity
radius} of $\M$, denoted $\rho(\M)$, is defined as the infimum of the set of
radii of the largest convex balls across the points of $\M$. Formally,
\[
\rho(\M)=\inf_{p\in\M}\sup\;\{r\geq0\mid\B_\M(p,s)
\text{ is convex for all }0<s<r\}.	
\]
\end{definition}
The compactness of $\M$ guarantees that $\rho(\M)$ is indeed positive;
see \cite[Proposition 95]{berger2007panoramic}. For example, the sphere $\S^n$
of radius $R$ has $\rho(\S^n)=\pi R/2$. The following remarkable result by
Hausmann states that $\M$ is homotopy equivalent to its Vietoris--Rips complex
for a scale smaller than the convexity radius.
\begin{theorem}[Hausmann's Theorem \cite{hausmann_1995}]\label{thm:hausmann} For
any $0<\beta<\rho(\M)$, the geometric complex of $\Ri_\beta(\M)$ is homotopy
equivalent to $\M$.
\end{theorem}
\begin{remark}
We remark that Hausmann defines the quantity $\rho(\M)$ (denoted $r(\M)$ by
Hausmann) slightly differently; see conditions (a)--(c) in \cite[Section
3]{hausmann_1995}. Nonetheless, the veracity of Hausmann's original result is
not compromised by the current substitution, since the implications from these
conditions used by Hausmann are still obtained using the current definition of
convexity radius. 	
\end{remark}

The definition of sectional curvatures of an abstract manifold $\M$ uses
involved machinery from Riemannian geometry. We skip the definition here,
suggesting the interested reader to call upon any graduate level textbook on the
subject, e.g., \cite[Chapter 9]{BishopRichardL2001Gom}. For a point $p\in\M$ and
(unit norm) vectors $u,v\in T_p(\M)$ the tangent space of $\M$, the sectional
curvature at $p$ along the plane spanned by $u,v$ is denoted by $\kappa_p(u,v)$.
It measures the Gaussian curvature at $p$ if $\M$ is a Euclidean
surface. The (embedded) sphere $\S^n$ of radius $R$ has a constant sectional
curvature of $1/R^2$. 

Let $\kappa(\M)\in\R$ denote the supremum of the set of sectional curvatures
$\kappa_p(u,v)$ across all $u,v$ and all $p$. Since $\M$ is compact, it can be
shown that $\kappa(\M)$ is finite; see \cite[p. 166]{BishopRichardL2001Gom} for
example. For the sake of simplifying the statements of our results, we
introduce:  
\begin{equation}\label{eqn:delta}
	\Delta(\M)=\begin{cases}
		\rho(\M),&\text{ if }\kappa(\M)\leq0 \\
		\min\left\{\rho(\M), \frac{\pi}{4\sqrt{\kappa(\M)}}\right\},&\text{ if }
		\kappa(\M)>0.
	\end{cases}
\end{equation}
The quantity $\frac{1}{\Delta(\M)}$ can be called the \emph{condition number} of
$\M$. The justification behind the name is that a manifold with a small
condition number is well-conditioned to be reconstructed; whereas, the recovery
of a manifold with a large condition number would require a large and extremely
dense sample.

\section{Abstract Manifold Reconstruction}\label{sec:abstract} Recall that $\M$ 
denotes an $n$--dimensional Riemannian manifold.
This section is devoted to the study of the Vietoris--Rips complexes of a metric
space $(S,d_S)$ that is close to $\M$ in the Gromov--Hausdorff distance (see
\defref{gh}). Our main homotopy equivalence result of this section is presented
in \thmref{gh-hom}. The proof of the result uses Jung's theorem as a very
important ingredient. We first discuss the classical Jung's theorem but in the
context of Riemannian manifolds.

\subsection{Jung's Theorem in Riemannian Manifolds}
For a compact subset $A\subset\M$, its diameter satisfies $\diam{A}<\infty$.
We define its \emph{circumradius} in $M$ to be
\[
\mathfrak{R}(A)\coloneqq \inf_{m\in \M}\max_{a\in A}d_\M(a,m).
\]
Intuitively, the circumradius is the radius of a smallest closed metric ball in
$M$ that contains $A$. The reader is warned that these terms may not exactly
match their usual meaning in plane geometry.

A point $m\in\M$ satisfying $\max_{a\in A}d_\M(a ,m)=\mathfrak{R}(A)$ is called
a \emph{circumcenter} of $A$, and is denoted by $\Theta(A)$. For $A$ compact, a
circumcenter may not uniquely exist. 

When $\M=\R^n$, however, the circumcenter exists uniquely. Moreover,
for a compact Euclidean subset $A\subset\R^d$, the classical Jung's theorem
\cite[Theorem 2.6]{danzer1963helly} states that
$\mathfrak{R}(A)\leq\sqrt{\frac{n}{2(n+1)}}\diam{A}$.

The result was further extended by Dekster in
\cite{Dekster1985AnEO,Dekster1995TheJT,Dekster1997}---first for compact subsets
of Riemannian manifolds with constant sectional curvatures, then for Alexadrov
spaces of curvature bounded above. A corollary in \cite[Section 2]{Dekster1997}
affirms that $\Theta(A)$ exists (possibly non-uniquely) for a compact $A$ if:
\begin{enumerate}[(i)]
\item $A$ is contained in the interior of a compact convex domain
$C^n\subset\M$, and
\item $\diam{C^n}<2\pi/(3\sqrt{\kappa(\M)})$ when $\kappa(\M)>0$.
\end{enumerate}
\begin{remark}
If $\diam{A}<\Delta(\M)$, we note that both the above conditions are satisfied.
In particular, one can choose $C^n$ to be the closed ball $\overline{\B(a,r)}$
for any $a\in A$ and $r$ with $\diam{A}<r<\Delta(\M)$. Here $\Delta(\M)$ is as
defined in \eqnref{delta}. 
\end{remark}\label{rem:jung}
Moreover, $\Theta(A)$ belongs to the interior of $C^n$, and we have the
following bound on the circumradius of $A$. 
\begin{theorem}[Extented Jung's Theorem \cite{Dekster1997}]\label{thm:Jung} Let
$\M$ be a compact, connected, $n$--dimensional manifold with the sectional
curvatures at each point bounded above by $\kappa\in\R$. For any compact
$A\subset\M$ with $\diam{A}<\Delta(\M)$, its circumcenter $\Theta(A)$ exists in
$\M$. Moreover, its diameter
\begin{equation}\label{eqn:Jung}
\diam{A}\geq\begin{cases}
\frac{2}{\sqrt{-\kappa}}\sinh^{-1}\left(\sqrt{\frac{n+1}{2n}}
\sinh\left(\sqrt{-\kappa}\;\mathfrak{R}(A)\right)\right),&\text{ for }\kappa<0\\
2\mathfrak{R}(A)\sqrt{\frac{n+1}{2n}},&\text{ for }\kappa=0 \\
\frac{2}{\sqrt{\kappa}}\sin^{-1}\left(\sqrt{\frac{n+1}{2n}}
\sin\left(\sqrt{\kappa}\;\mathfrak{R}(A)\right)\right),
&\text{ for }\kappa>0\text{ and}\\&\mathfrak{R}(A)\in
\left[0,\frac{\pi}{2\sqrt{\kappa}}\right]
\end{cases}
\end{equation}
\end{theorem}
Utilizing the above result, we show the following key result bounding the
circumradius. See Appendix for a proof.
\begin{proposition}[Circumradius]\label{prop:circum} For any compact
$A\subset\M$ with $\diam{A}<\Delta(\M)$, its diameter satisfies
\[
	\diam{A}\geq\frac{4}{3}\mathfrak{R}(A).
\]
\end{proposition}
We work with the constant $\frac{4}{3}$ in the paper for convenience.
It can be improved, however, when an upper bound on the sectional curvature is known
(such as for some Lie groups). 

We now note the following important proposition.
\begin{proposition}[Circumcenters of Subsets]\label{prop:rad} If $A$ is a
compact subset of $\M$ with $\diam{A}<\Delta(\M)$, then for any non-empty subset
$B\subseteq A$, we have
\[d_\M\left(\Theta(B),\Theta(A)\right) \leq\frac{3}{4}\diam{A}.\]
\end{proposition}
\begin{proof}
If follows from \thmref{Jung} that $\Theta(A)$ and $\Theta(B)$ exist. The
definition of $\mathfrak{R}(A)$ then implies that
$A\subseteq\overline{\B_\M(\Theta(A),\mathfrak{R}(A))}$, the closed metric ball
of radius $\mathfrak{R}(A)$ centered at $\Theta(A)$. From \propref{circum}
again, we have 
\[
\mathfrak{R}(A)\leq\frac{3}{4}\diam{A}<\frac{3}{4}\Delta(\M).
\]
For any arbitrary $\alpha$ with $0<\alpha<\Delta(\M)/4$, the convex domain $C^n$
in \remref{jung} can then be chosen to be the closed ball
$\overline{\B_\M(\Theta(A),\mathfrak{R}(A)+\alpha)}$. Since $A\subseteq B$, $B$
lies in the interior of $C^n$, so does $\Theta(B)$. Therefore,
\[
d_\M(\Theta(B),\Theta(A))<\mathfrak{R}(A)+\alpha\leq\frac{3}{4}\diam{A}+\alpha.
\]
Since $\alpha$ is arbitrary, the result follows.
\end{proof}

\subsection{Homotopy Equivalence}
We now assume that $(S,d_S)$ is a compact metric space such that the
Gromov--Hausdorff distance $d_{GH}(S,\M)<\zeta\beta$ for some $\beta>0$ and
$0<\zeta<1/2$. From the definition of the Gromov--Hausdorff distance
(\defref{gh}), then there exists a correspondence $\C\in\C(\M,S)$ with
$\mathrm{dist}(\C)<2\zeta\beta$. The correspondence induces a (possibly
non-continuous and non-unique) vertex map $\phi:\M\to S$ such that
$(p,\phi(p))\in\C$ for all $p\in\M$. The vertex map $\phi$ extends to a
simplicial map $\phi$:
\begin{equation}\label{eqn:phi}
\Ri_{(1-2\zeta)\beta}(\M)\xrightarrow{\quad\phi\quad}\Ri_\beta(S).
\end{equation}
To see that $\phi$ is a simplicial map, take an $l$--simplex
$\sigma_l=[p_0,p_1,\ldots,p_l]$ in $\Ri_{(1-2\zeta)\beta}(\M)$. By the
construction of the Vietoris--Rips complex, we must have
$d_\M(p_i,p_j)<(1-2\zeta)\beta$ for any $0\leq i,j\leq l$. Since
$\mathrm{dist}(\C)<2\zeta\beta$ and $(p,\phi(p))\in\C$ for all $p\in\M$, we
have
\[
d_S(\phi(p_i),\phi(p_j))\leq d_\M(p_i,p_j)+2\zeta\beta
<(1-2\zeta)\beta+2\zeta\beta=\beta.
\]
So, the image $\phi(\sigma_l)=[\phi(p_0),\phi(p_1),\ldots,\phi(p_l)]$ is a
simplex of $\Ri_\beta(S)$.

In the rest of the section, we show that $\phi$ induces a homotopy equivalence
on the respective geometric complexes. First, we show in the following lemma
that the simplicial map induces a surjective homomorphism on all homotopy
groups.
\begin{lemma}[Surjectivity]\label{lem:gh-sur} Let $(S,d_S)$ be a compact metric
space and $\beta>0$ a number such that  
\[
	\frac{1}{\zeta}d_{GH}(\M,S)<\beta<\frac{1}{1+2\zeta}\Delta(\M)
\]
for some $0<\zeta\leq1/14$. Then for any $m\geq0$, the simplicial map
$\phi:\Ri_{(1-2\zeta)\beta}(\M)\to\Ri_\beta(S)$ (as defined in \eqnref{phi})
induces a surjective homomorphism on the $m$--th homotopy group.
\end{lemma}
\begin{proof}
As observed in \propref{path-connected}, both $\Ri_{(1-2\zeta)\beta}(\M)$ and
$\Ri_\beta(S)$ are path-connected. So, the result holds for $m=0$. 
	
For $m\geq1$, let us take an abstract simplicial complex $\K$ such that
$\mod{\K}$ is a triangulation of the $m$-dimensional sphere $\S^m$. Note that
$\K$ is a pure $m$--complex. In order to show surjectivity of $\mod{\phi}_*$ on
$\pi_m\left(\mod{\Ri_{(1-2\zeta)\beta}(\M)}\right)$, we start with a simplicial
map $g:\K\map\Ri_{\beta}(S)$, and argue that there must exist a simplicial map
$\widetilde{g}:\sd{\K}\map\Ri_{(1-2\zeta)\beta}(\M)$ such that the following
diagram commutes up to homotopy:
\begin{equation}\label{eqn:diag}
\begin{tikzpicture} [baseline=(current  bounding  box.center)]
	\node (k1) at (-2,0) {$\mod{\Ri_{(1-2\zeta)\beta}(\M)}$};
	\node (k2) at (2,0) {$\mod{\Ri_\beta(S)}$};
	\node (k3) at (2,-2) {$\mod{\K}$};
	\node (k4) at (-2,-2) {$\mod{\sd{\K}}$};
	\draw[map] (k1) to node[auto] {$\mod{\phi}$} (k2);
	\draw[map,swap] (k3) to node[auto,swap] {$\mod{g}$} (k2);
	\draw[map,dashed] (k4) to node[auto,swap] {$\mod{\widetilde{g}}$} (k1);
	\draw[map, <-, dashed] (k4) to node[auto] {$h^{-1}$} (k3);
\end{tikzpicture}
\end{equation}
where the linear homeomorphism $h:\mod{\sd{\K}}\map\mod{\K}$ maps each vertex of
$\sd{\K}$ to the corresponding point of $\mod{\K}$ as discussed in
\subsecref{subdiv}. We remark here that it suffices to consider an arbitrary
simplicial map $g$ due to the simplicial approximation theorem \cite[Theorem
16.5]{MUNK}.

We note that each vertex of $\sd{\K}$ is the barycenter, $\widehat{\sigma}$, of
a simplex $\sigma$ of $\K$. In order to construct the simplicial map
$\widetilde{g}:\sd{\K}\map\Ri_{(1-2\zeta)\beta}(\M)$, we define it on the
vertices of $\sd{\K}$ first, and prove that the vertex map extends to a
simplicial map. 

Let $\sigma_l=[v_0,v_1,\ldots,v_l]$ be an $l$--simplex of $\K$. Since $g$ is a
simplicial map, then the image $g(\sigma_l)=[g(v_0),g(v_1),\ldots,g(v_l)]$ is a
simplex of $\Ri_\beta(S)$, hence a subset of $S$ with
$\diam[S]{g(\sigma_l)}<\beta$. For each $0\leq j\leq l$, there exists $p_j\in\M$
such that $(p_j,g(v_j))\in\C$. Recall that $\C$ is a correspondence with
distortion $\mathrm{dist}(\C)<2\zeta\beta$ as already fixed right above
\eqnref{phi}. We denote $\sigma_j':=[p_0,p_1,\ldots,p_j]$ for $0\leq j\leq l$.
We note for later that the diameter of $\sigma_l'$ is less than $\Delta(\M)$: 
\begin{equation}\label{eqn:sigma}
\diam{\sigma_j'}\leq\diam[S]{g(\sigma_j)}+2\zeta\beta
<\beta+2\zeta\beta=(1+2\zeta)\beta<\Delta(\M).
\end{equation}
We then define the vertex map
\[
	\widetilde{g}(\widehat{\sigma}_l)\eqdef\Theta(\sigma_l'),
\]
where $\Theta(\sigma_l')\in\M$ is a circumcenter of $\sigma_l'$. Due to the
diameter bound in \eqnref{sigma}, \thmref{Jung} implies that a circumcenter of
$\sigma_l'$ exists. 

To see that $\widetilde{g}$ extends to a simplicial map, we consider a typical
$l$--simplex $\tau_l=[\widehat{\sigma}_0,\ldots,\widehat{\sigma}_l]$, of
$\sd{\K}$, where $\sigma_{i-1}\prec\sigma_i$ for $1\leq i\leq l$ and
$\sigma_i\in\K$. Now,
\begin{align*}
	\diam{\widetilde{g}(\tau_l)}
	&=\diam{[\Theta(\sigma_0'),\Theta(\sigma_1'),\ldots,\Theta(\sigma_l')]} \\
	&=\max_{0\leq i<j\leq l}\{d_\M(\Theta(\sigma_i'),\Theta(\sigma_j'))\} \\
	&\leq\max_{0\leq j\leq l}\left\{ \left(\frac{3}{4}\right)\diam{\sigma_j'}
	\right\}, \\
	&\quad\quad\quad\text{ by \propref{rad} as }\diam{\sigma_j'}<\Delta(\M) \\
	&=\frac{3}{4}\diam{\sigma_l'} \\
	&<\frac{3}{4}(1+2\zeta)\beta,\text{ from \eqnref{sigma}} \\
	&=(1-2\zeta)\beta-(1-14\zeta)\beta/4 \\
	&\leq(1-2\zeta)\beta,\text{ since }\zeta\leq1/14.
\end{align*}
So, $\widetilde{g}(\tau_l)$ is a simplex of $\Ri_{(1-2\zeta)\beta}(\M)$. This
implies that $\widetilde{g}$ is a simplicial map. 
	
We lastly invoke \propref{homotopy} to show that the diagram commutes up to
homotopy. We need to argue that the simplicial maps $g$ and
$(\phi\circ\widetilde{g})$ satisfy the conditions of \propref{homotopy}:
\begin{enumerate}[(a)]
\item For any vertex $v\in\K$, 
\[
(\phi\circ\widetilde{g})(v)=g(v).
\]

\item For any simplex $\sigma_m=[v_0,v_1,\ldots,v_m]$ of $\K$, we have for $0\leq
j\leq m$:
\begin{align*}
d_S(g(v_j),(\phi\circ\widetilde{g})(\widehat{\sigma}_m))	
&=d_S(g(v_j),\phi(\Theta(\sigma_m'))) \\
&\leq d_\M\left(p_j,\Theta(\sigma_m')\right)+2\zeta\beta,
\text{ since }(p_j,g(v_j))\in\C \\
&\leq\frac{3}{4}\diam[\M]{\sigma_m'}+2\zeta\beta,
\text{ by \propref{rad} as }p_j=\Theta({p_j}) \\
&<\frac{3}{4}(1+2\zeta)\beta+2\zeta\beta,
\text{ from \eqnref{sigma}} \\
&=\beta-(1-14\zeta)\beta/4 \\
&\leq\beta,\text{ since }\zeta\leq1/14.
\end{align*}
\end{enumerate}
So, $g(\sigma_m)\cup(\phi\circ\widetilde{g})(\widehat{\sigma}_m)$ is a simplex
of $\Ri_\beta(S)$. 

Therefore, \propref{homotopy} implies that the diagram commutes. Since
$\mod{\K}=\S^m$ and $g$ is arbitrary, we conclude that $\mod{\phi}$ induces a
surjective homomorphism.
\end{proof}

\begin{remark}\label{rem:basepoint} For the description and computation of
homotopy groups, the consideration of basepoint is deliberately ignored
throughout this paper. This is justified, as the considered scale parameters are
such that all the Vietoris--Rips complexes used here are path-connected. We
prove the claim in \propref{path-connected}.	
\end{remark}

\ghhom
\begin{proof}
Since $d_{GH}(\M,S)<\zeta\beta$, let us assume that $\C\in\C(\M,S)$ is a
correspondence with $\mathrm{dist}(\C)<2\zeta\beta$. As a result, we have the
following chain of simplicial maps
\[
\Ri_{(1-2\zeta)\beta}(\M)\xrightarrow{\quad\phi\quad}
\Ri_\beta(S)\xrightarrow{\quad\psi\quad}\Ri_{(1+2\zeta)\beta}(\M),
\]
such that $(p,\phi(p))\in\C$ for all $p\in\M$ and $(\psi(x),x)\in\C$ for all
$x\in S$. There is also the natural inclusion
$\Ri_{(1-2\zeta)\beta}(\M)\xhookrightarrow{\quad
\iota\quad}\Ri_{(1+2\zeta)\beta}(\M)$. We first claim that $(\psi\circ\phi)$ and
$\iota$ are contiguous. To prove the claim, take an $l$--simplex
$\sigma_l=[p_0,p_1,\ldots,p_l]$ in $\Ri_{(1-2\zeta)\beta}(\M)$. So,
$d_\M(p_i,p_j)<(1-2\zeta)\beta$ for all $0\leq i,j\leq l$. We then have
\begin{align*}
d_\M((\psi\circ\phi)(p_i),p_j) &=d_\M(\psi(\phi(p_i)),p_j) \\
&\leq d_S(\phi(p_i),\phi(p_j))+2\zeta\beta \\
&\leq d_\M(p_i,p_j)+2\zeta\beta+2\zeta\beta \\
&<(1-2\zeta)\beta+4\zeta\beta=(1+2\zeta)\beta.
\end{align*}
This implies that $(\psi\circ\phi)(\sigma_l)\cup\iota(\sigma_l)$ is a simplex of
$\Ri_{(1+2\zeta)\beta}(\M)$. Since $\sigma_l$ is an arbitrary simplex, the
simplicial maps $(\psi\circ\phi)$ and $\iota$ are contiguous. Consequently, the
maps $\mod{\psi\circ\phi}$ and $\mod{\iota}$ are homotopic.
		
Since $(1+2\zeta)\beta<\Delta(\M)\leq\rho(\M)$, \thmref{hausmann} implies that
there exist homotopy equivalences $T_1,T_2$ such that the following diagram
commutes (up to homotopy):
\begin{equation*}
\begin{tikzpicture} [baseline=(current  bounding  box.center)]
	\node (k1) at (-2,0) {$\mod{\Ri_{(1-2\zeta)\beta}(\M)}$};
	\node (k2) at (2,0) {$\mod{\Ri_{(1+2\zeta)\beta}(\M)}$};
	\node (k3) at (0,-2) {$\M$};
	\draw[rinclusion] (k1) to node[auto] {$\mod{\iota}$} (k2);
	\draw[map,swap] (k1) to node[auto] {$T_1$} (k3);
	\draw[map,swap] (k2) to node[auto,swap] {$T_2$} (k3);
\end{tikzpicture}
\end{equation*}
So, $\mod{\iota}$ is also a homotopy equivalence. Hence, the induced
homomorphism $\mod{\iota}_*$ on the homotopy groups is an isomorphism. On the
other hand, we already have $\mod{\iota}\simeq\mod{\psi\circ\phi}$. Therefore,
the induced homomorphism $\left(\mod{\psi}_*\circ\mod{\phi}_*\right)$ is also an
isomorphism, implying that $\mod{\phi}_*$ is an injective homomorphism on
$\pi_m\left(\Ri_{(1-2\zeta)\beta}(\M)\right)$. The surjectivity of
$\mod{\phi}_*$ comes from \lemref{gh-sur}. For any $m\geq0$, therefore,
\[
	\mod{\phi}_*:\pi_m\left(\mod{\Ri_{(1-2\zeta)\beta}(\M)}\right)
	\map\pi_m\left(\mod{\Ri_\beta(S)}\right).	
\]
is an isomorphism. 

It follows from Whitehead's theorem that $\mod{\phi}$ is a homotopy equivalence.
Since $\mod{\Ri_{(1-2\zeta)\beta}(\M)}$ is homotopy equivalent to $\M$, we
conclude that $\mod{\Ri_\beta(S)}\simeq\M$.
\end{proof}

\section{Hausdorff Reconstruction of Euclidean Submanifolds}
\label{sec:euclidean}%
Let $\M\subset\R^d$ be a closed, connected (smoothly embedded) Euclidean
submanifold. In this section, we consider the topological reconstruction of $\M$
from a Euclidean subset $S\subset\R^d$ that is close to $\M$ in the Hausdorff
distance. \thmref{h-hom} is the main homotopy equivalence result of the section.
To provide the sampling conditions, we use the reach of $\M$, which we define
first. 

The \emph{medial axis} of a compact subset $X\subset\R^d$ is the set of points
$y\in\R^d$ such that there are at least two (distinct) points $x_1,x_2\in X$
with
\[
	\norm{x_1-y}=\norm{x_2-y}=\min_{x\in X}{\norm{x-y}}.
\]
\begin{definition}[Reach]\label{def:reach} The \emph{reach} of $X$, denoted by
$\tau(X)$, is the minimum of the set of distances between a point of $X$ and a
point on its medial axis. 
\end{definition}
It can be shown that $\tau(\M)$ is positive for a smoothly embedded closed
submanifold $\M$. As we show in the next proposition, the reach controls both
the sectional curvatures and the convexity radius of $\M$. 

Fix a point $p\in\M$. Let $T_p(\M)$ and $T^\perp_p(\M)$ denote, respectively,
the tangent and normal space of $\M$ at $p$. It can be shown that a symmetric,
bilinear form $B(u,v):T_p(\M)\times T_p(\M)\to T^\perp_p(\M)$ exists, called the
\emph{second fundamental form} at $p$. More details can be found in any standard
text, e.g., \cite[Chapter 6]{alma991019430369703276}. For any orthonormal
vectors $u,v\in T_p(\M)$, the \emph{sectional curvature} at $p$ along the plane
generated by $u,v$ is defined as
\begin{equation}\label{eqn:K}
	\kappa(u,v)\eqdef\langle B(u,u), B(v,v)\rangle - \|B(u,v)\|^2,
\end{equation}
where $\langle\cdot,\cdot\rangle$ and $\norm{\cdot}$ denote the standard
Euclidean inner product and norm, respectively. For any normal vector $\eta\in
T_p^\perp(\M)$, one can define a symmetric, bilinear form
\[
	B_\eta(u,v)\eqdef\langle\eta,B(u,v)\rangle,\;\; u,v\in T_p(\M).
\]

Let us denote by $L_\eta:T_p(\M)\to T_p(\M)$ the linear, self-adjoint operator
associated to the bilinear form $B_\eta(u,v)$, i.e., $B_\eta(u,v)=\big\langle
u,L_\eta(v)\big\rangle$ for all $u,v\in T_p(\M)$. Using an important result from
\cite{SMALE} connecting the norm of $L_\eta$ with the reach $\tau(\M)$, we list
the following consequences. A proof is presented in Appendix.
\begin{proposition}\label{prop:B}
Let $p\in\M$ be any point and $u,v\in T_p(\M)$ unit norm. Then,
\begin{enumerate}[(i)]
\item $\norm{B(u,v)}\leq1/\tau(\M)$,
\item $-1/\tau(\M)^2\leq\kappa(\M)\leq1/\tau(\M)^2$,
\item $\rho(\M)\geq\pi\tau(\M)/2$, and
\item $\Delta(\M)\geq\pi\tau(\M)/4$.
\end{enumerate}
\end{proposition}
Finally, we obtain the following important bound on the distortion of a pair of
points on $\M$. See Appendix for a proof.
\begin{proposition}\label{prop:distortion} Let $1<\xi<2$ and $p,q\in\M$ be such
that $\norm{p-q}\leq2\left(\frac{\xi-1}{\xi^2}\right)\tau(\M)$. Then,
\[
	d_\M(p,q)\leq\xi\norm{p-q}.
\]
\end{proposition}
We now assume that $S\subset\R^d$ is a compact subset and $\beta>0$ a number
such that the Hausdorff distance $d_{H}(\M,S)<\zeta\beta$ for some
$0<\zeta<1/2$. There is a (possibly non-continuous and non-unique) vertex map
$\phi:\M\map S$ such that $\norm{p-\phi(p)}<\zeta\beta$ for all $p\in\M$. Note
that the definition of $\phi$ is analogous to the map \eqnref{phi} defined in
\secref{abstract}. Indeed, the map here is induced by the nearest neighbor
correspondence.

Now, the vertex map $\phi$ extends to a simplicial map:
\begin{equation}\label{eqn:phi-h}
	\Ri_{(1-2\zeta)\beta}(\M)
	\xrightarrow{\quad\phi\quad}\Ri_\beta(S).
\end{equation}
To see that $\phi$ is a simplicial map, take an $l$--simplex
$\sigma_l=[p_0,p_1,\ldots,p_l]$ in $\Ri_{(1-2\zeta)\beta}(\M)$. By the
construction of the Vietoris--Rips complex, we must have
$d_\M(p_i,p_j)<(1-2\zeta)\beta$ for any $0\leq i,j\leq l$. Using the triangle
inequality, we get 
\begin{align*}
\norm{\phi(p_i)-\phi(p_j)}
&\leq\norm{\phi(p_i)-p_i}+\norm{p_i-p_j}+\norm{p_j-\phi(p_j)} \\
&<\zeta\beta+d_\M(p_i,p_j)+\zeta\beta \\
&<\zeta\beta+(1-2\zeta)\beta+\zeta\beta=\beta.
\end{align*}
So, the image $\phi(\sigma_l)=[\phi(p_0),\phi(p_1),\ldots,\phi(p_l)]$ is a
simplex of $\Ri_\beta(S)$.

The following lemma proves that the simplicial map $\phi$ defined in
\eqnref{phi-h} induces a surjective homomorphism on the homotopy groups. 
\begin{lemma}[Surjectivity]\label{lem:h-sur} Let $\M\subset\R^d$ be a closed,
connected submanifold, $S\subset\R^d$ a compact subset, and $\beta>0$ a number
such that
\[
\frac{1}{\zeta}d_{H}(\M,S)<\beta\leq\frac{3(1+2\zeta)(1-14\zeta)}{8(1-2\zeta)^2}
\tau(\M)
\]
for some $0<\zeta<1/14$. Then for any $m\geq0$, the simplicial map
$\phi:\Ri_{(1-2\zeta)\beta}(\M)\to\Ri_\beta(S)$ (as defined in \eqnref{phi-h})
induces a surjective homomorphism on the $m$--th homotopy group.
\end{lemma}
\begin{proof}
Due to \propref{path-connected}, the complexes $\Ri_{(1-2\zeta)\beta}(\M)$ and
$\Ri_\beta(S)$ are path-connected. So, the result holds for $m=0$.

For $m\geq1$, let us take an abstract simplicial complex $\K$ such that
$\mod{\K}$ is a triangulation of the $m$--dimensional sphere $\S^m$. In order to
show surjectivity of $\mod{\phi}_*$, we start with a simplicial map
$g:\K\to\Ri_\beta(S)$, and argue that there must exist a simplicial map
$\widetilde{g}:\sd{\K}\map\Ri_{(1-2\zeta)\beta}(\M)$ such that the following
diagram commutes up to homotopy:
\begin{equation}\label{eqn:diag-1}
\begin{tikzpicture} [baseline=(current  bounding  box.center)]
	\node (k1) at (-2,0) {$\mod{\Ri_{(1-2\zeta)\beta}(\M)}$};
	\node (k2) at (2,0) {$\mod{\Ri_\beta(S)}$};
	\node (k3) at (2,-2) {$\mod{\K}$};
	\node (k4) at (-2,-2) {$\mod{\sd{\K}}$};
	\draw[map] (k1) to node[auto] {$\mod{\phi}$} (k2);
	\draw[map,swap] (k3) to node[auto] {$\mod{g}$} (k2);
	\draw[map,dashed] (k4) to node[auto] {$\mod{\widetilde{g}}$} (k1);
	\draw[map, <-, dashed] (k4) to node[auto] {$h^{-1}$} (k3);
\end{tikzpicture}
\end{equation}
where the linear homeomorphism $h:\mod{\sd{\K}}\map\mod{\K}$ maps each vertex
of $\sd{\K}$ to the corresponding point of $\mod{\K}$.

We first note that each vertex of $\sd{\K}$ is the barycenter,
$\widehat{\sigma}_l$, of an $l$--simplex $\sigma_l$ of $\K$. In order to
construct the simplicial map $\widetilde{g}:\sd{\K}\map\Ri_\beta(\M)$, we define
it on the vertices $\sd{\K}$ first, and prove that the vertex map extends to a
simplicial map. 

Let $\sigma_l=[v_0,v_1,\ldots,v_l]$ be an $l$--simplex of $\K$. Since $g$ is a
simplicial map, we have that the image
$g(\sigma_l)=[g(v_0),g(v_1),\ldots,g(v_l)]$ is a subset of $S$ with
$\diam[S]{g(\sigma_l)}<\beta$. There is a corresponding subset
$\sigma_l'=[p_0,p_1,\ldots,p_l]\subset\M$ with $\norm{p_j-g(v_j)}<\zeta\beta$
for $0\leq j\leq l$. Choose $\xi=\frac{4(1-2\zeta)}{3(1+2\zeta)}$. Since
$0<\zeta<1/14$, we observe that $1<\xi<2$. Note from our assumption that 
\[
	\norm{p_i-p_j}<\beta
	\leq\frac{3(1+2\zeta)(1-14\zeta)}{8(1-2\zeta)^2}\tau(\M)
	=2\left(\frac{\xi-1}{\xi^2}\right)\tau(\M).
\]
By \propref{distortion}, we then have
\begin{align*}
d_\M(p_i,p_j)&\leq\xi\norm{p_i-p_j} \\
&=\frac{4(1-2\zeta)}{3(1+2\zeta)}\norm{p_i-p_j}\\
&\leq\frac{4(1-2\zeta)}{3(1+2\zeta)}(\norm{p_i-g(v_i)}+\norm{g(v_i)-g(v_j)}+\norm{g(v_j)-p_j}) \\
&<\frac{4(1-2\zeta)}{3(1+2\zeta)}(\zeta\beta+\norm{g(v_i)-g(v_j)} + \zeta\beta) \\
&<\frac{4(1-2\zeta)}{3(1+2\zeta)}(\beta+2\zeta\beta)=\frac{4}{3}(1-2\zeta)\beta.
\end{align*}
For any $0\leq j\leq l$, define $\sigma_j':=[p_0,p_1,\ldots,p_j]$. Therefore, the
diameter 
\begin{equation}\label{eqn:sigma-1}
\diam{\sigma_j'}<\frac{4}{3}(1-2\zeta)\beta.
\end{equation}  
Moreover, due to our assumption on the upper bound on $\beta$:
\begin{align}\label{eqn:beta}
\diam{\sigma_j'}&<\frac{4}{3}(1-2\zeta)\beta\nonumber \\
&\leq\frac{4}{3}(1-2\zeta)\frac{3(1+2\zeta)(1-14\zeta)}{8(1-2\zeta)^2}\tau(\M) \\
&=\frac{1}{2}\left(\frac{1-12\zeta-28\zeta^2}{1-2\zeta}\right)\tau(\M)
\nonumber \\
&<\frac{1}{2}\cdot1\cdot\tau(\M),\text{ since }\zeta>0\nonumber \\
&<\pi\tau(\M)/4\leq\Delta(\M),\text{ from \propref{B}}.\nonumber	
\end{align}
So, \propref{circum} implies that $\Theta(\sigma_l')$ exists. We define 
\[
\widetilde{g}(\widehat{\sigma}_l)\eqdef\Theta(\sigma_l').
\]
To see that $\widetilde{g}$ extends to a simplicial map, consider a typical
$l$--simplex,
$\tau_l=[\widehat{\sigma}_0,\widehat{\sigma}_1,\ldots,\widehat{\sigma}_l]$, of
$\sd{\K}$, where $\sigma_i\prec\sigma_{i+1}$ for $0\leq i\leq l-1$ and
$\sigma_i\in\K$. Now,
\begin{align*}
\diam{\widetilde{g}(\tau_l)}
&=\diam{[\Theta(\sigma_0'),\Theta(\sigma_1'),
\ldots,\Theta(\sigma_l')]}\\
&=\max_{0\leq i<j\leq l}\{d_\M(\Theta(\sigma_i'),
\Theta(\sigma_j'))\} \\
&\leq\max_{0\leq j\leq l}\left\{\frac{3}{4}\diam{\sigma_j'}
\right\},\text{ by \propref{rad}}\\
&\leq\frac{3}{4}\diam{\sigma_l'} \\
&<\frac{3}{4}\cdot\frac{4}{3}(1-2\zeta)\beta,\text{ from \eqnref{sigma-1}} \\
&=(1-2\zeta)\beta.
\end{align*}
So, $\widetilde{g}(\tau_l)$ is a simplex of $\Ri_{(1-2\zeta)\beta}(\M)$. This
implies that $\widetilde{g}$ is a simplicial map. 
	
We invoke \propref{homotopy} to show that Diagram \eqnref{diag-1} commutes up to
homotopy. We need to argue that the simplicial maps $g$ and
$(\phi\circ\widetilde{g})$ satisfy the conditions of \propref{homotopy}:
\begin{enumerate}[(a)]
\item For any vertex $v\in\K$, 
\[
	(\phi\circ\widetilde g)(v)=g(v).
\]

\item For any simplex $\sigma_m=[v_0,v_1,\ldots,v_m]$ of $\K$, we have for
$0\leq j\leq m$:
\begin{align*}
\norm{g(v_j)-(\phi\circ\widetilde g)(\widehat{\sigma}_m)}
&<\norm{p_j-\Theta(\sigma'_m)}+2\zeta\beta\\
&\leq d_\M(p_j,\Theta(\sigma_m'))+2\zeta\beta \\
&=\frac{3}{4}\diam{\sigma_m'}+2\zeta\beta,
\text{ by \propref{rad} as }\Theta(p_j)=p_j \\
&<(1-2\zeta)\beta+2\zeta\beta=\beta.
\end{align*}
So, $g(\sigma_m)\cup(\phi\circ\widetilde g)(\Theta(\sigma_m))$ is a simplex of
$\Ri_\beta(S)$.
\end{enumerate}
Therefore, \propref{homotopy} implies that the diagram commutes. Since
$\mod{\K}=\S^m$ and $g$ is arbitrary, we conclude that $\phi$ induces a
surjective homomorphism on the $m$--th homotopy group.
\end{proof}

\hhom
\begin{proof}
Since $d_{H}(\M,S)<\zeta\beta$, we have the following chain of simplicial maps
\[
	\Ri_{(1-2\zeta)\beta}(\M)\xrightarrow{\quad\phi\quad}
	\Ri_\beta(S)\xrightarrow{\quad\psi\quad}
	\Ri_{\frac{4}{3}(1-2\zeta)\beta}(\M),
\]
such that $\norm{p-\phi(p)}<\zeta\beta$ for all $p\in\M$ and
$\norm{\psi(x)-x}<\zeta\beta$ for all $x\in S$. The second simplicial map can be
justified by \eqnref{sigma-1}.
		   
There is also the natural inclusion
$\Ri_{(1-2\zeta)\beta}(\M)\xhookrightarrow{\quad
\iota\quad}\Ri_{\frac{4}{3}(1-2\zeta)\beta}(\M)$. We first claim that
$(\psi\circ\phi)$ and $\iota$ are contiguous. To prove the claim, take an
$l$--simplex $\sigma_l=[p_0,p_1,\ldots,p_l]$ of $\Ri_{(1-2\zeta)\beta}(\M)$.
Then, $d_\M(p_i,p_j)<(1-2\zeta)\beta$ for all $0\leq i,j\leq l$. So, we have
\begin{align*}
d_\M((\psi\circ\phi)(p_i),p_j) &=d_\M(\psi(\phi(p_i)),p_j) \\
&\leq\frac{4(1-2\zeta)}{3(1+2\zeta)}\norm{\psi(\phi(p_i))-p_j},
\text{ by \propref{distortion} for }\xi=\frac{4(1-2\zeta)}{3(1+2\zeta)} \\
&<\frac{4(1-2\zeta)}{3(1+2\zeta)}(\norm{\phi(p_i)-\phi(p_j)}+2\zeta\beta) \\
&<\frac{4(1-2\zeta)}{3(1+2\zeta)}(\norm{p_i-p_j}+2\zeta\beta+2\zeta\beta) \\
&<\frac{4(1-2\zeta)}{3(1+2\zeta)}((1-2\zeta)\beta+4\zeta\beta)
=\frac{4}{3}(1-2\zeta)\beta.
\end{align*}
This implies that $(\psi\circ\phi)(\sigma_l)\cup\iota(\sigma_l)$ is a simplex of
$\Ri_{\frac{4}{3}(1-2\zeta)\beta}(\M)$. Since $\sigma_l$ is an arbitrary
simplex, the simplicial maps $(\psi\circ\phi)$ and $\iota$ are contiguous.
Consequently, the maps $\mod{\psi\circ\phi}$ and $\mod{\iota}$ are homotopic.
		
As already argued in \eqnref{beta}, we have
$\frac{4}{3}(1-2\zeta)\beta<\Delta(\M)\leq\rho(\M)$. So, \thmref{hausmann}
implies there exist homotopy equivalences $T_1,T_2$ such that the following
diagram commutes (up to homotopy):
\begin{equation*}
\begin{tikzpicture} [baseline=(current  bounding  box.center)]
\node (k1) at (-3,0) {$\mod{\Ri_{(1-2\zeta)\beta}(\M)}$};
\node (k2) at (3,0) {$\mod{\Ri_{\frac{4}{3}(1-2\zeta)\beta}(\M)}$};
\node (k3) at (0,-2) {$\M$};
\draw[rinclusion] (k1) to node[auto] {$\mod{\iota}$} (k2);
\draw[map,swap] (k1) to node[auto] {$T_1$} (k3);
\draw[map,swap] (k2) to node[auto,swap] {$T_2$} (k3);
\end{tikzpicture}
\end{equation*}
So, $\mod{\iota}$ is also a homotopy equivalence. Hence, the induced
homomorphism $\mod{\iota}_*$ on the homotopy groups is an isomorphism. On the
other hand, $\mod{\iota}\simeq\mod{\psi\circ\phi}$. Therefore, the induced
homomorphism $(\mod{\psi}_*\circ\mod{\phi}_*)$ is an isomorphism, implying that
$\mod{\phi}_*$ is an injective homomorphism.

By \lemref{h-sur} for any $m\geq0$, 
\[
	\mod{\phi}_*:\pi_m(\mod{\Ri_{(1-2\zeta)\beta}(\M)})
	\map\pi_m(\mod{\Ri_\beta(S)}).	
\]
is an isomorphism. By Whitehead's theorem, we have that $\mod{\phi}$ is a
homotopy equivalence. Since $\mod{\Ri_{(1-2\zeta)\beta}(\M)}$ is homotopy
equivalent to $\M$, we conclude that $\mod{\Ri_\beta(S)}\simeq\M$.
\end{proof}	

\section{Conclusion}
The current work provides satisfactory answers to the quest of recovering a
closed Riemannian manifold $\M$ from the Vietoris--Rips complexes of a compact
metric space $S$ close to it---both in the Gromov--Hausdorff and Hausdorff
distance. The study sparks a number of intriguing future research directions.
Although we provide a homotopy equivalent recovery of a Euclidean submanifold,
the resulting complex $\Ri_\beta(S)$, being very high-dimensional, does not
produce a natural embedding for the reconstruction. Consequently, our result for
submanifold reconstruction does not lend itself well to recovering the geometry.
Since $S$ is a subset of $\R^d$, one may consider the shadow (as defined by
Chambers et al. \cite{Chambers2010}) of the resulting Vietoris--Rips complex as
a \emph{geometric reconstruction} of $\M$. As pointed out in \cite[Proposition
5.3]{Chambers2010}, the shadow of a complex is notorious for being topologically
unfaithful. When the Hausdorff distance between $S$ and $\M$ is very small, however, we
conjecture to have homotopy equivalent shadow of $\Ri_\beta(S)$, hence providing
a homotopy equivalent reconstruction of $\M$ with an embedding in the same
ambient space. 

\paragraph{Acknowledgments}
The author would like to thank Henry Adams for providing constructive feedback
on the manuscript, and for hosting the author at the University of Florida,
Gainesville during the summer of 2023.

\bibliographystyle{plain}
\bibliography{main}

\begin{appendices}

\section{Additional Proofs}
\begin{proof}[Proof of \propref{circum}] We divide the proof into the following
two cases depending on the sign of $\kappa:=\kappa(\M)$.

\paragraph{Case 1 ($\kappa\leq0$)}
From \eqnref{Jung}, we get 
\[
	\diam{A}\geq2\mathfrak{R}(A)\sqrt{\frac{n+1}{2n}}\geq2\mathfrak{R}(A)\frac{1}{\sqrt{2}}\geq\frac{4}{3}\mathfrak{R}(A).
\]
	
\paragraph{Case 2 ($\kappa>0$)}
First note that the closed ball $\overline{\B_\M(p,\diam{A})}$ for any $p\in A$
contains $A$ entirely, implying that $\mathfrak{R}(A)\leq\diam{A}$. Since
$\diam{A}<\Delta(\M)\leq\frac{\pi}{4\sqrt{\kappa}}$, we therefore have
$\mathfrak{R}(A)\in\left[0,\frac{\pi}{4\sqrt{\kappa}}\right]$. 

Let us now define the function
$J(r)\eqdef\frac{2}{\sqrt{\kappa}}\sin^{-1}\left(\sqrt{\frac{n+1}{2n}}\sin{\sqrt{\kappa}r}\right)$,
and show that $J(r)/r$ is a strictly decreasing function for
$r\in\left(0,\frac{\pi}{4\sqrt{\kappa}}\right)$. We take its derivative:
\begin{equation}\label{eqn:J}
\frac{d}{dr}(J(r)/r)=\frac{rJ'(r)-J(r)}{r^2}=\frac{f(r)}{r^2},
\end{equation}
where $f(r)\eqdef rJ'(r)-J(r)$. We observe that $f(r)$ is strictly decreasing,
since
\[
f'(r)=rJ''(r)=r\frac{2\sqrt{\kappa}\sqrt{\frac{n+1}{2n}}\sin{\sqrt{\kappa}r}}{\left(1-\frac{n+1}{2n}\sin^2{\sqrt{\kappa}r}\right)^{3/2}}\left(\frac{1-n}{2n}\right)<0
\]
for $r\in\left(0,\frac{\pi}{4\sqrt{\kappa}}\right)$ and $n\geq2$. Then it
follows from $f(0)=0$ that $f(r)<0$ on
$r\in\left(0,\frac{\pi}{4\sqrt{\kappa}}\right)$. When $n=1$, we remark that the
result can be established using Case 1, since $\M$ must be flat. 
	
So, \eqnref{J} implies that the function $J(r)/r$ is a strictly decreasing
function for $r\in\left(0,\frac{\pi}{4\sqrt{\kappa}}\right)$. Its minimum, as a
result, is attained at $r=\frac{\pi}{4\sqrt{\kappa}}$. So, for
$r\in\left(0,\frac{\pi}{4\sqrt{\kappa}}\right]$ we get
\[
\frac{J(r)}{r}\geq\frac{2}{\sqrt{\kappa}}\sin^{-1}\left(\sqrt{\frac{n+1}{2n}}\sin{\frac{\pi}{4}}\right)\frac{4\sqrt{\kappa}}{\pi}
\geq\frac{8}{\pi}\sin^{-1}\left(\sqrt{\frac{1}{2}}\sin{\frac{\pi}{4}}\right)
=\frac{8}{\pi}\frac{\pi}{6}=\frac{4}{3}.
\] 
On the other hand, \eqnref{Jung} implies that $\diam{A}\geq J(\mathfrak{R}(A))$
for $\mathfrak{R}(A)\in\left[0,\frac{\pi}{4\sqrt{\kappa}}\right]$. Hence,
$\diam{A}\geq\frac{4}{3}\mathfrak{R}(A)$ as desired.
	
\end{proof}

\begin{proposition}[Path-connectedness]\label{prop:path-connected} Let $(S,d_S)$
be a metric space and $\beta>0$ a number such that $d_{GH}(S,\M)<\beta$, then
for any positive $\alpha$, the geometric complex of $\Ri_{\alpha+2\beta}(S)$ is
path-connected. The result also holds if $M\subset\R^d$ and $S\subset\R^d$ with
$d_H(\M,S)<\beta$.
\end{proposition}
\begin{proof}
Let $a,b\in S$, then there exist points $a',b'\in\M$ such that
$(a',a),(b',b)\in\C$, where $\C\in\C(\M,S)$ is a correspondence with
$\mathrm{dist}(\C)<2\beta$. Since $\M$ is assumed to be path-connected, so is
$\Ri_\alpha(\M)$. As a result, there exists a sequence
$\{p_i'\}_{i=0}^{m+1}\subset\M$ forming a path in $\Ri_\alpha(\M)$ joining $a'$
and $b'$. In other words, $p_0'=a'$, $p_{m+1}'=b'$, and
$d_\M(p_i',p_{i+1}')<\alpha$ for $0\leq i\leq m$. There is also a corresponding
sequence $\{x_i\}_{i=0}^{m+1}\subset S$ such that $x_0=a$, $x_{m+1}=b$, and
$(p_i',x_i)\in\C$ for all $i$. We note that
\[
	d_S(x_i,x_{i+1})\leq d_\M(p_i',p_{i+1}')+2\beta<\alpha+2\beta.	
\]
So, the sequence $\{x_i\}$ produces a path in $\Ri_{\alpha+2\beta}(S)$ joining
$a$ and $b$. We conclude that the geometric complex of $\Ri_{\alpha+2\beta}(S)$
is path-connected.

The Hausdorff case uses the exact same argument.
\end{proof}

\begin{proof}[Proof of \propref{B}]
In \cite[Proposition 6.1]{SMALE}, it has been shown that
$\norm{L_\eta}\leq1/\tau(\M)$ for any unit norm $\eta\in T^\perp_p(\M)$.
	
\begin{enumerate}[(i)]
\item Take $\eta=B(u,v)/\norm{B(u,v)}$, to get $\norm{B(u,v)}=\big\langle \eta,
B(u,v)\big\rangle=\big\langle u, L_\eta(v)\big\rangle$. Since $u,v$ are unit
norm, applying the Cauchy-Schwarz inequality we get
\[
\norm{B(u,v)}
\leq\norm{u}\norm{L_\eta}\norm{v}=\norm{L_\eta}\leq\frac{1}{\tau(\M)}.
\]
	
\item Use \eqnref{K}, then apply the Cauchy-Schwarz inequality: 
\[
\kappa_p(u,v)\leq\langle B(u,u),B(v,v)\rangle\leq\norm{B(u,u)}\cdot\norm{B(v,v)}
\leq\frac{1}{\tau(\M)^2}.
\] So, $\kappa(\M)\leq1/\tau(\M)^2$.
	
Similarly from \eqnref{K}, we get
\[
-\kappa_p(u,v)\leq\norm{B(u,v)}^2\leq\frac{1}{\tau(\M)^2}.
\]
So, $\kappa(\M)\geq-1/\tau(\M)^2$. 
	
\item From \cite[Proposition 95]{berger2007panoramic}, we have
\[
\rho(\M)\geq\begin{cases}
\min\left\{\frac{\pi}{2\sqrt{\kappa(\M)}},\frac{1}{2}\;\mathrm{inj}(\M)\right\},
&\text{ for }\kappa(\M)>0\\
\frac{1}{2}\;\mathrm{inj}(\M),&\text{ for }\kappa(\M)\leq0.
\end{cases}
\] 
See \cite[Definition 23]{berger2007panoramic} for the definition of injectivity
radius $\mathrm{inj}(\M)$. On the other hand, \cite[Corollary 4]{Alexander2006}
shows that $\mathrm{inj}(\M)\geq\pi\tau(\M)$. This, combining with (ii), implies
that $\rho(\M)\geq\pi\tau(\M)/2$.
	
\item Follows directly from \eqnref{delta} due to (iii). 
\end{enumerate}
\end{proof}

\begin{proof}[Proof of \propref{distortion}]
Let $\gamma:[0,L]\to\M$ be an arc-length parametrized, unit speed,
length-minimizing geodesic joining $p$ and $q$, where $L=d_\M(p,q)$. From the
fundamental theorem of calculus, we have
\[
	\gamma(L)-\gamma(0)=\int_0^L \gamma'(t)dt,
\]
and
\[
	\gamma'(t)=\gamma'(0)+\int_0^t \gamma''(s)ds.
\]
Combining, we can write
\[
	\gamma(L)-\gamma(0)=\int_0^L\left[\gamma'(0)+\int_0^t \gamma''(s)ds\right]dt
	=L\gamma'(0) + \int_0^L\left[\int_0^t \gamma''(s) ds\right]dt.
\]
From the triangle inequality and the fact that $\gamma$ is unit speed,
\[
	\norm{\gamma(L)-\gamma(0)}\geq L\norm{\gamma'(0)} - 
	\norm{\int_0^L\left[\int_0^t \gamma''(s) ds\right]dt}
	\geq L-\int_0^L\left[\int_0^t\norm{\gamma''(s)} ds\right]dt.
\]
We note that $\norm{p-q}=\norm{\gamma(L)-\gamma(0)}$ and $L=d_\M(p,q)$. Since
$\gamma$ is geodesic, $\gamma''(s)=B\left(\gamma'(s),\gamma'(s)\right)$ for all
$s$. So,
\[
	\norm{p-q}
	\geq L-\int_0^L\left[\int_0^t\norm{B(\gamma'(s),\gamma'(s))} ds\right]dt.
\]
On the other hand, $\norm{B(u,u)}\leq1/\tau(\M)$ from \propref{B} implies that
\[
	\norm{p-q}\geq L-\frac{1}{\tau(\M)}\int_0^L\left[\int_0^t ds\right]dt
	=L-\frac{1}{\tau(\M)}\frac{L^2}{2}
	=d_\M(p,q)-\frac{1}{2\tau(\M)}[d_\M(p,q)]^2.
\]
The inequality is satisfied only if $\norm{p-q}\leq\tau(\M)/2$ and
\begin{equation}\label{eqn:tau}
	d_\M(p,q)\leq\tau(\M)-\tau(\M)\sqrt{1-\frac{2\norm{p-q}}{\tau(\M)}}.
\end{equation}
Let $f(r)\eqdef\tau(\M)-\tau(\M)\sqrt{1-\frac{2r}{\tau(\M)}}$ for $0<
r<2\left(\frac{\xi-1}{\xi^2}\right)\tau(\M)$. We note by taking derivative that
$\frac{f(r)}{r}$ is an increasing function. The maximum is attained for
$r=2\left(\frac{\xi-1}{\xi^2}\right)\tau(\M)$. Therefore,
\[
\frac{f(r)}{r}
\leq\frac{1-\frac{\sqrt{(2-\xi)^2}}{\xi}}{\frac{2(\xi-1)}{\xi^2}}
=\xi\frac{\xi-\sqrt{(2-\xi)^2}}{2(\xi-1)}
=\xi\frac{\xi-(2-\xi)}{2(\xi-1)}
=\xi,\text{ since }1<\xi<2.
\]
Hence, \eqnref{tau} implies 
\[
	d_\M(p,q)\leq f(\norm{p-q})
	\leq\xi\norm{p-q},\text{ since }\norm{p-q}\leq2\left(\frac{\xi-1}{\xi^2}\right)\tau(\M)\leq\tau(\M)/2.
\]
\end{proof}

\begin{proposition}[Commuting Diagram]\label{prop:homotopy} Let $\K$ be a pure
	$m$--complex and $\L$ a flag complex. Let
	$f:~\K\map\L$ and $g:\sd{\K}\map\L$ be simplicial maps such that 
	\begin{enumerate}[(a)]
		\item $g(v)=f(v)$ for every vertex $v$ of $\K$,
		\item $f(\sigma)\cup g(\bc{\sigma})$ is a simplex of $\L$ whenever
		$\sigma$ is a simplex of $\K$.
	\end{enumerate}
	Then, the following diagram commutes up to homotopy: 
	\begin{equation*}
		\begin{tikzpicture} [baseline=(current  bounding  box.center)]
			\node (k1) at (-2,-2) {$\mod{\sd{\K}}$};
			\node (k2) at (2,-2) {$\mod{\K}$};
			\node (k3) at (0,0) {$\mod{\L}$};
			\draw[map] (k2) to node[auto,swap] {$h^{-1}$} (k1);
			\draw[map,swap] (k1) to node[auto,swap] {$\mod{g}$} (k3);
			\draw[map,swap] (k2) to node[auto] {$\mod{f}$} (k3);
		\end{tikzpicture}
	\end{equation*}
	where $h$ is a linear homeomorphism sending each vertex of $\sd{\K}$ to the
    corresponding point of $\mod{\K}$.
\end{proposition}
\begin{proof}
We show that the maps $\mod{g}$ and $\left(\mod{f}\circ h\right)$ are homotopic
by constructing an explicit homotopy $H:~\mod{\sd{\K}}\times[0,1]\map\mod{\L}$
with $H(\cdot,0)=g(\cdot)$ and $H(\cdot,1)=\left(\mod{f}\circ h\right)(\cdot)$.
The commutativity of the diagram then follows, since $h$ is a homeomorphism. 
	
The complex $\K$ is taken to be $m$--dimensional. Without any loss of
generality, we can assume that every point of $\sd{\K}$ belongs to an
$m$--simplex. Take an arbitrary $x\in\mod{\sd{\K}}$. We can write $x$ in its
barycentric coordinates as $x=\sum_{i=0}^m\zeta_i\widehat{\sigma}_i$, where
$\sigma_m=[a_0,a_1,\ldots,a_m]$ is an $m$--simplex of $\K$ and
$\sigma_i=[a_0,a_1,\ldots,a_i]$ for $0\leq i\leq m$. Consider a partition 
\[0=t_0<t_1<\ldots<t_{m}=1\] of $[0,1]$. We first define the homotopy
$H(\cdot,t)$ at $t=t_i$ for $i=0,1,\ldots,m$, and then show that we can
\emph{interpolate} $H$ continuously (using the straight-line homotopy) for any
$t\in[t_{i-1},t_i]$. For any $0\leq i\leq m$, define 
\begin{equation}\label{eqn:homotopy}
		H(x,t_i)=\sum_{j=0}^i\zeta_j\left(\mod{f}\circ
		h\right)(\widehat{\sigma}_j)+\sum_{j=i+1}^m\zeta_j\mod{g}(\widehat{\sigma}_j).
\end{equation} 
From the above definition, we note for $i=0$ that 
\begin{align*}
H(x,0)=\zeta_0\left(\mod{f}\circ
h\right)(\widehat{\sigma}_0)+\sum_{j=1}^m\zeta_j\mod{g}(\widehat{\sigma}_j)
&=\zeta_0\mod{g}(\widehat{\sigma}_0)
+\sum_{j=1}^m\zeta_j\mod{g}(\widehat{\sigma}_j)
\\&=\sum_{j=0}^m\zeta_j\mod{g}(\widehat{\sigma}_j)=\mod{g}(x).	
\end{align*}
The second equality is due to condition (a) and fact that $\sigma_0$ is a vertex
of $\K$. Also, for $i=m$ we note that 
\[
H(x,1)=\sum_{j=0}^m\zeta_j\left(\mod{f}\circ h\right)(\widehat{\sigma}_j)
=\left(\mod{f}\circ h\right)\left(\sum_{j=0}^m\zeta_j\widehat{\sigma}_j\right)
=\left(\mod{f}\circ h\right)(x).
\] 
The second equality is due to the fact that $h$ is a linear homeomorphism.
	
We now fix $0\leq i\leq m$, and extend the definition of $H(x,t)$ for any
$t\in[t_{i-1},t_i]$ by using the straight-line joining $H(x,t_{i-1})$ and
$H(x,t_{i})$. Such an extension is justified, we show now both $H(x,t_{i-1})$
and $H(x,t_{i})$ belong to $\mod{\eta}$ for some simplex $\eta$ of $\L$. 

Let us take $\eta$ to be the set $f(\sigma_i) \cup
g([\widehat{\sigma}_i,\widehat{\sigma}_{i+1},\ldots,\widehat{\sigma}_m])$. Since
$g$ is a simplicial map, $\eta$ can be written as $f(\sigma_i)\cup
[g(\widehat{\sigma}_i),g(\widehat{\sigma}_{i+1}),\ldots,g(\widehat{\sigma}_m)]$.
Now, note from condition (b) that $f(\sigma_j)\cup g(\widehat{\sigma}_j)\in\L$
for any $i\leq j\leq m$. Since $f$ is a simplicial map, $f(\sigma_i)\cup g(\widehat{\sigma}_j)$
is a face of $f(\sigma_j)\cup g(\widehat{\sigma}_j)$, hence a simplex of $\L$.
Consequently, any pair of elements in the set $\eta=f(\sigma_i)\cup
[g(\widehat{\sigma}_i),g(\widehat{\sigma}_{i+1}),\ldots,g(\widehat{\sigma}_m)]$
is a simplex of $\L$. Since $\L$ is a flag complex, it implies that $\eta$
is a simplex of $\L$. 

From the definition \eqnref{homotopy} and the fact that
$h(\widehat{\sigma}_{l})\in\mod{\sigma_i}$ for all $0\leq l\leq i$, we conclude
that both $H(x,t_{i-1})$ and $H(x,t_{i})$ belong to $\mod{\eta}$. So, $H(x,t)$
is well-defined and continuous for $t\in[t_{i-1},t_i]$. Moreover, the images
agree at the endpoints as noted from \eqnref{homotopy}. Therefore, $H$ defines
the desired homotopy.
\end{proof}

\end{appendices}

\end{document}